\documentclass[10pt]{article}
\usepackage{mathrsfs}
\usepackage{amsfonts}
\usepackage{amsthm,amsmath,amssymb,anysize}
\newtheorem{lemma}{Lemma}[section]
\newtheorem{theorem}[lemma]{Theorem}

\newtheorem{definition}[lemma]{Definition}

\usepackage[numbers,sort&compress]{natbib}
\marginsize{3.6cm}{3.6cm}{3.6cm}{3cm}
\numberwithin{equation}{section}

\title{\textsf{ Restricted cohomology of restricted   Lie superalgebras}}
\author{\textsc{Jixia Yuan,$^{1,2,}$}\footnote{Supported by   NNSF of China (11601135), Natural Science Foundation of Heilongjiang Province of China (QC2017002) and Project funded by China Postdoctoral Science Foundation(2018M630311)}\;\;\textsc{Liangyun Chen$^{1,}$}\footnote{Correspondence: chenly640@nenu.edu.cn. (L. Chen);  Supported by  NNSF of China (11771069  and  12071405)}\;\;  \textsc{ and Yan Cao$^{3,}$}\footnote{Supported by  NNSF of China (11801121), Natural Science Foundation of Heilongjiang Province of China (QC2018006) } \\
  \\
  \textit{1. School of Mathematical and Statistics},
  \textit{Northeast Normal University} \\
  \textit{Changchun 130024, China}\\
\\
  \ \ \textit{2. School of Mathematical Sciences},
  \textit{Heilongjiang University} \\
  \textit{Harbin 150080, China}\\
  \\
  \ \ \textit{3.  Department of Mathematics},
  \textit{Harbin University of Science and Technology} \\
  \textit{Harbin 150080, China}
  }
\date{ }
\begin{document}
\maketitle
\begin{quotation}
\noindent\textbf{Abstract.} Suppose the ground   field $\mathbb{F}$ is an algebraically closed  field characteristic of   $p>2$. In this paper,  we investigate the restricted cohomology theory of restricted Lie superalgebras. Algebraic interpretations of  low dimensional  restricted cohomology of restricted Lie superalgebra are given.  We show that there is a family of restricted model filiform  Lie superalgebra $L_{p,p}^{\lambda}$ structures parameterized   by elements $\lambda\in \mathbb{F}^{p}.$   We explicitly describe both the $1$-dimensional ordinary and restricted cohomology superspaces of $L_{p,p}^{\lambda}$ with coefficients   in the $1$-dimensional trivial module and show that  these superspaces are equal. We also describe the $2$-dimensional ordinary and restricted cohomology superspaces  of $L_{p,p}^{\lambda}$ with coefficients   in the $1$-dimensional trivial module and show that  these superspaces are unequal.
\\



  \end{quotation}

  \setcounter{section}{0}
\section{Introduction}

As a natural generalization of Lie algebras, Lie superalgebras play an important role in the theoretical physics and mathematics.
Let $L=L_{\bar{0}}\oplus L_{\bar{1}}$ be a restricted Lie superalgebra. Then $L_{\bar{0}}$ is a restricted Lie algebra and $L_{\bar{1}}$ is a restricted $L_{\bar{0}}$-module. Therefore, restricted Lie superalgebras  play a central role in the theory of modular Lie superalgebras, just as in the modular Lie algebra situation.

Cohomology is a very important  tool  in the
research of    topology, smooth vector fields, holomorphic functions, etc.
Cohomology  theory is from the study of the topology of Lie groups and
vector fields on Lie groups by  Cartan. The standard complex of Lie algebra  was originally constructed by
Chevalley and Eilenberg in \cite{ce}.  Many conclusions on cohomology of  Lie algebras were given in \cite{gf1}, etc.
Hochschild first considered the cohomology theory of restricted Lie algebras in \cite{h}. In dissertation  \cite{e}, Evans  improved on the results in \cite{h} and  obtained a cochain complex that is capable of making computations. In recent years,   restricted cohomology theory of restricted Lie algebras have aroused the interest of a great many researchers, see  \cite{ef1,ef2,ef3,efp}. Shortly after the birth of supersymmetry  most basic constructions and results of the classical theory of Lie algebras are generalized to the case of Lie superalgebras. Fuks and Leites in \cite{fl} calculated the cohomology  of the classical Lie superalgebras with trivial coefficients.  In \cite{sz},  Scheunert and Zhang  introduced in detail the concepts of cohomology of  Lie superalgebra.  In \cite{sz1},  Su and Zhang  explicitly computed the $1, 2$-dimensional cohomology  of the classical Lie superalgebras $\mathfrak{sl}_{m|n}$ and $\mathfrak{osp}_{2|2n}$ with coefficients in the finite dimensional irreducible modules and the Kac modules.
In addition, the works of Boe, Kujawa, Nakano, Bagci, Lehrer, Poletaeva etc also promoted the development of cohomology theory of Lie superalgebras, see \cite{b,bkn,lnz}.
The papers of Iwai and Shimada \cite{is},   May \cite{m2} and Priddy \cite{p} from the 1960s   address
the problem of constructing a resolution for computing the cohomology of restricted
graded Lie algebras. Strictly speaking, these earlier works were written in the context of $\mathbb{Z}$-graded Lie algebras, also the same
methods may apply to Lie superalgebras by reducing all of the gradings modulo $2$. For a discussion of this earlier
work that is specifically in the context of Lie superalgebras, see \cite{d}.  The resolutions constructed in these earlier works are not entirely
explicit (owing to the inductive nature of the constructions), but the resolutions are
relatively explicit in low degrees. In recent years,   cohomology theory of restricted Lie superalgebras have aroused the interest of some researchers, see   \cite{b1}.

In the study of the reducibility of the varieties of nilpotent Lie algebras, Vergne introduced the concept of filiform Lie algebras, see \cite{v}. Since then, the study of the filiform Lie algebras, especially the model filiform Lie algebra, has become an important subject.  As what happens in the Lie case, (model) filiform Lie superalgebra is   an important subject of Lie superalgebra. Many conclusions on cohomology of the model filiform Lie superalgebra were given in \cite{g,bgkn,ly}.

In this paper, we are interested in the  cohomology theory of restricted Lie superalgebras. We give algebraic interpretations of   restricted cohomology of restricted Lie superalgebra $L$: The superspace of all restricted outer superderivations is equal to $1$-dimensional restricted cohomology of  $L$ with coefficients in the adjoint module; Let superspaces $M, N$ be two restricted $L$-modules,  then  equivalence classes of extensions of $N$ by $M$ are in one to one correspondence with $1$-dimensional restricted  cohomology of  $L$ with coefficients in the  module $ \mathrm{Hom}_{\mathbb{F}}(N, M)$; Equivalence classes of
 restricted central extensions of $L$  are in one to one correspondence  with  $2$-dimensional restricted  cohomology of  $L$ with coefficients in the  trivial module.  We show that there is a family $L_{p,p}^{\lambda}$ of restricted model filiform  Lie superalgebra structures parameterized   by elements $\lambda\in \mathbb{F}^{\lambda}.$    We explicitly describe both the $1, 2$-dimensional ordinary and restricted  cohomology   of $L_{p,p}^{\lambda}$ with coefficients   in the $1$-dimensional trivial module.

This  paper is organized as follows. In Section 2, we give some necessary concepts and notations. In   Section 3, we  give algebraic interpretations of   restricted cohomology of restricted Lie superalgebra. In Section 4, we determine both the $1, 2$-dimensional ordinary and restricted  cohomology of restricted model filiform  Lie superalgebra with coefficients   in the $1$-dimensional trivial module.

\section{Preliminaries}
Let $\mathbb{F}$ be an algebraically closed field of characteristics $p>2$.  All unspecified vector superspaces are taken over $\mathbb{F}$. Suppose
$\mathbb{Z}_2:= \{\bar{0},\bar{1}\}$ is the additive group of order 2.  Suppose $ \mathfrak{A}:= \mathfrak{A}_{\bar{0}}\oplus  \mathfrak{A}_{\bar{1}}$ is a
 superalgebra.  We write $|x|:=\theta$
for the \textit{parity} of a $\mathbb{Z}_{2}$-homogeneous element $x\in  \mathfrak{A}_{\theta}, $
$\theta\in \mathbb{Z}_{2}.$

\subsection{Restricted Lie superalgebras}
Let us  recall the definitions of Lie superalgebras and restricted Lie superalgebras \cite{k1,zl}.
\begin{definition}
A Lie superalgebra is a vector superspace $L=L_{\bar{0}}\oplus
L_{\bar{1}}$ with an even bilinear   mapping
$[\cdot,\cdot]: L\times L\longrightarrow L$   satisfying the following axioms:
\begin{eqnarray*}
&& [x,y]=-(-1)^{|x||y|}[y,x],\\
&&[x,[y,z]]=[[x,y],z]+(-1)^{|x||y|}[y,[x,z]]
\end{eqnarray*}
for all $x, y, z\in L.$
\end{definition}
\begin{definition}
A restricted  Lie superalgebras is   a Lie superalgebra $L=L_{\bar{0}}\oplus L_{\bar{1}}$    together with a map $[p]: L_{\bar{0}}\longrightarrow L_{\bar{0}}$, denoted by $x\longmapsto x^{[p]}$, that satisfies:
\begin{itemize}
  \item [(1)] $(L_{\bar{0}}, [p])$ is a restricted  Lie  algebra,
  \item [(2)] $L_{\bar{1}}$ is a restricted  $L_{\bar{0}}$-module.
\end{itemize}
\end{definition}

\subsection{Ordinary  and restricted cohomology of restricted Lie superalgebras}
Let us  give some definitions relative to cohomology of Lie superalgebras\cite{sz} and   restricted cohomology of Lie superalgebras. For details  of ordinary cohomologies we refer to \cite{b2,fl,sz}, which are the foundational papers on cohonology of Lie superalgebras.  By \cite{d,e,is,m2,p}, we  derive the types of  $1, 2$-dimensional  restricted  cohomology of restricted Lie superalgebras.

 Let $L$ be a Lie superalgebra and $M=M_{\bar{0}}\oplus M_{\bar{1}}$ be a $L$-module. For $q\geq0$, we let $\bigwedge^{q}L$  be the $q$-th super-exterior product of $L$, that is  $\bigwedge^{q}L$ is the $q$-fold tensor product of $L$ modulo the $L$-submodule generated by the elements of the form:
$$x_{1}\otimes\cdots\otimes x_{k}\otimes x_{k+1}\otimes\cdots \otimes x_{q}+(-1)^{|x_{k}||x_{k+1}|}x_{1}\otimes\cdots\otimes x_{k+1}\otimes x_{k}\otimes\cdots \otimes x_{q}$$
for $x_{1},\ldots, x_{q}\in L.$ Set
$$C^{q}(L; M)=\mathrm{Hom}_{\mathbb{F}}(\bigwedge^{q}L, M).$$
There spaces $C^{q}(L; M)$ are naturally $\mathbb{Z}_{2}$-graded.
Note that
$$C^{0}(L; M)=\mathrm{Hom}_{\mathbb{F}}(\mathbb{F}, M)\cong M.$$

 Let
$d^{q}:C^{q}(L; M)\longrightarrow C^{q+1}(L; M)$  be given by the formula:
\begin{eqnarray}\label{e2.1}\nonumber
d^{q}(\varphi)(x_{1},\ldots, x_{q+1})&=&\sum_{1\leq i< j\leq q+1}(-1)^{\sigma_{ij}(x_{1},\ldots, x_{q+1})}\varphi([x_{i},x_{j}], x_{1},\ldots,\hat{x}_{i},\ldots \hat{x}_{j},\ldots, x_{q+1})\\
&&+\sum_{i=1}^{q+1}(-1)^{\gamma_{i}(x_{1},\ldots, x_{q+1},\varphi)}x_{i}\varphi(x_{1},\ldots,\hat{x}_{i},\ldots, x_{q+1}),
\end{eqnarray}
where $x_{1},\ldots, x_{q+1}\in L,$ $\varphi\in C^{q}(L; M)$ and
$$\sigma_{ij}(x_{1},\ldots, x_{q+1})=i+j+|x_{i}|(|x_{1}|+\cdots+|x_{i-1}|)+|x_{j}|(|x_{1}|+\cdots+|x_{j-1}|+|x_{i}|),$$
$$\gamma_{i}(x_{1},\ldots, x_{q+1},\varphi)=i+1+|x_{i}|(|x_{1}|+\cdots+|x_{i-1}|+|\varphi|).$$
A direct verification shows that $d^{q}d^{q-1}=0$ and $|d^{q}|=\bar{0}.$ The elements of  kernel of $d^{q}$ are called $q$-dimensional cocycles and  the elements of  image of $d^{q-1}$ are called $q$-dimensional  coboundaries. We will denote the $q$-dimensional cocycles and coboundaries by $Z^{q}(L; M)$ and $B^{q}(L; M)$,  respectively.
\begin{definition}
We called $H^{q}(L; M)=Z^{q}(L; M)/ B^{q}(L; M)$ is an ordinary cohomology  of $L$ with coefficients   in the module $M.$
\end{definition}

Let $q\leq 1,$ define $C_{*}^{q}(L; M)=C^{q}(L; M)$. If $\varphi\in C^{2}(L; M)$ and $\omega: L_{\bar{0}}\longrightarrow M$, we say that $\omega$ has the $*$-property with respect to $\varphi$ if for all $\lambda\in \mathbb{F}$ and all $x, y\in L_{\bar{0}}$
\begin{itemize}
  \item [(i)] $\omega(\lambda x)=\lambda^{p}\omega(x)$,
  \item [(ii)] $\omega(x+y)=\omega(x)+\omega(y)+\sum\limits_{x_{i}=x \; or\; y \atop x_{1}=x, x_{2}=y}\frac{1}{\sharp(x)}\sum_{k=0}^{p-2} (-1)^{k} x_{p}\cdots x_{p-k+1}\varphi([x_{1},\ldots,x_{p-k-1}], x_{p-k}),$
      where $\sharp(x)$ is the number of factors $x_{i}$ equal to  $x$ and
      $$[x_{1},\ldots,x_{p-k-1}]=[[[\cdots[[x_{1}, x_{2}],x_{3}],\cdots],  x_{p-k-2}],x_{p-k-1}].$$

\end{itemize}
Set
$$C_{*}^{2}(L; M)=\{(\varphi, \omega)\mid \varphi\in C^{2}(L; M), \omega: L_{\bar{0}}\longrightarrow M  \;\mbox{has the}\; *\mbox{-property w.r.t.}\; \varphi \}.$$

If $\alpha\in C^{3}(L; M)$ and $\beta: L_{\bar{0}}\times L_{\bar{0}}\longrightarrow M$, we say that $\beta$ has the $**$-property w.r.t.  $\alpha$ if for all $\lambda\in \mathbb{F}$ and all $x, y, y_{1}, y_{2}\in L_{\bar{0}}$
\begin{itemize}
  \item [(i)] $\beta(x, y)$ is linear with respect to $x$,
    \item [(ii)] $\beta(x, \lambda y)=\lambda^{p}\beta(x, y)$,
  \item [(iii)]
  \begin{eqnarray*}
  \beta(x, y_{1}+y_{2})&=&\beta(x,y_{2})+\beta(x, y_{2})-\sum\limits_{h_{i}=y_{1} \; or\; y_{2} \atop h_{1}=y_{1}, h_{2}=y_{2}}\frac{1}{\sharp(y_{1})}\sum_{j=0}^{p-2} (-1)^{j}\sum_{k=1}^{j}C_{j}^{k} h_{p}\cdots h_{p-k+1}\cdot\\
  &&\alpha([x,h_{p-k},\ldots, h_{p-j+1}],[h_{1},\ldots,h_{p-j-1}], h_{p-j}).
  \end{eqnarray*}

\end{itemize}
Set
$$C_{*}^{3}(L; M)=\{(\alpha, \beta)\mid \alpha\in C^{3}(L; M), \beta: L_{\bar{0}}\times L_{\bar{0}}\longrightarrow M  \;\mbox{has the}\; **\mbox{-property w.r.t.}\;\alpha\}.$$
Let $q\leq 3$, the $\mathbb{Z}_{2}$-grading of $C^{q}(L; M)$ is inherited by
$$C_{*}^{q}(L; M)=C_{*}^{q}(L; M)_{\bar{0}}\oplus C_{*}^{q}(L; M)_{\bar{1}},$$
where
$$C_{*}^{q}(L; M)_{\theta}=\{(\alpha, \beta)\in C_{*}^{q}(L; M)\mid |\alpha|=\theta \} \; \mbox{for all}\; \theta\in \mathbb{Z}_{2}.$$
Now, an element $\varphi\in C_{*}^{1}(L; M)$  induces a map $\mathrm{ind}^{1}(\varphi): L_{\bar{0}}\longrightarrow M$ by the formula
 \begin{eqnarray}\label{e2.2}
 \mathrm{ind}^{1}(\varphi)(x)=\varphi(x^{[p]})-x^{p-1}\varphi(x) \;\mbox{for all}\; x\in L_{\bar{0}}.
  \end{eqnarray}
Since  $d^{1}\varphi\mid_{L_{\bar{0}}\times L_{\bar{0}}}\in C^{2}(L_{\bar{0}}; M)$, then  the  \cite[Lemma 3.7]{e} shows that $\mathrm{ind}^{1}(\varphi)$ satisfies the $*$-property w.r.t. $d^{1}\varphi.$ An element $(\alpha, \beta)\in C_{*}^{2}(L; M)$  induces a map $\mathrm{ind}^{2}(\alpha, \beta): L_{\bar{0}}\times L_{\bar{0}}\longrightarrow M$ by the formula
 \begin{eqnarray}\label{e2.3}
 \mathrm{ind}^{2}(\alpha,\beta)(x,y)=\alpha(x, y^{[p]})-\sum_{i+j=p-1}(-1)^{i}y^{i}\alpha ([x,\underbrace{y,\ldots, y}_{j}],y)+x\beta(y) \;\mbox{for all}\; x, y\in L_{\bar{0}}.
  \end{eqnarray}
Since $d^{2}\alpha\mid_{L_{\bar{0}}\times L_{\bar{0}}\times L_{\bar{0}}}\in C^{3}(L_{\bar{0}}; M)$,  then  the  \cite[Lemma 3.10]{e} shows that $\mathrm{ind}^{2}(\alpha, \beta)$ satisfies the $**$-property w.r.t. $d^{2}\alpha.$
 The restricted differentials are defined by
  \begin{eqnarray}\label{e2.4}
  &&d_{*}^{0}: C_{*}^{0}(L; M)\longrightarrow C_{*}^{1}(L; M)\;\;\;\; d_{*}^{0}=d^{0},\\\label{e2.5}
  &&d_{*}^{1}: C_{*}^{1}(L; M)\longrightarrow C_{*}^{2}(L; M)\;\;\; d_{*}^{1}(\varphi)=(d^{1}\varphi, \mathrm{ind}^{1}(\varphi)),\\\label{e2.6}
  &&d_{*}^{2}: C_{*}^{2}(L; M)\longrightarrow C_{*}^{3}(L; M)\;\;\; d_{*}^{2}(\alpha, \beta)=(d^{2}\alpha, \mathrm{ind}^{2}(\alpha,\beta)).
  \end{eqnarray}
Let $q\leq2.$ We can  show that $d_{*}^{q}d_{*}^{q-1}=0$ and $|d_{*}^{q}|=\bar{0}$.  The elements of  kernel of $d_{*}^{q}$ are called $q$-dimensional  restricted cocycles and  the elements of  image of $d_{*}^{q-1}$ are called $q$-dimensional restricted  coboundaries. We will denote the $q$-dimensional restricted  cocycles and restricted  coboundaries by $Z_{*}^{q}(L; M)$ and $B_{*}^{q}(L; M)$,  respectively.
\begin{definition}
We called $H_{*}^{q}(L; M)=Z_{*}^{q}(L; M)/ B_{*}^{q}(L; M)$ is $q$-dimensional restricted  cohomology  of $L$ with  coefficients   in the module $M.$
\end{definition}
Note that $H^{0}(L; M)=H_{*}^{0}(L; M).$

\section{Algebraic interpretations}
In this section, we give the algebraic interpretations of low dimensional restricted cohomology  of restricted Lie superalgebras and show
that equivalence classes of these objects are naturally corresponding to the  restricted
cohomology  of restricted Lie superalgebras defined above.
\begin{definition}
Let $L$ be a restricted Lie superalgebra. An  homogeneous linear map $D: L\longrightarrow L$ is called a restricted superderivation of $L$ with parity $|D|$ if for all $x, y\in L$, $z\in L_{\bar{0}}$
  \begin{eqnarray}\label{e3.1}
  &&D([x, y])=(-1)^{|D||x|}[x, D(y)]+[D(x), y],\\\label{e3.2}
  &&D(z^{[p]})=(\mathrm{ad}z)^{p-1}D(z).
  \end{eqnarray}
\end{definition}
Write $(\mathrm{Der}_{\mathrm{res.}})_{\bar{0}}(L)$ (resp. $(\mathrm{Der}_{\mathrm{res.}})_{\bar{1}}(L)$) for the set of all restricted superderivations with parity $\bar{0}$ (resp.  $\bar{1}$) of $L$.  Denote
$$\mathrm{Der}_{\mathrm{res.}}(L)=(\mathrm{Der}_{\mathrm{res.}})_{\bar{0}}(L)\oplus (\mathrm{Der}_{\mathrm{res.}})_{\bar{1}}(L).$$
For  $x\in L$, we have $\mathrm{ad}x\in \mathrm{Der}_{\mathrm{res.}}(L)$ and then $\mathrm{ad}L\vartriangleleft \mathrm{Der}_{\mathrm{res.}}(L).$ We will call restricted superderivations of the form $\mathrm{ad}x$ inner and elements of the $\mathrm{Der}_{\mathrm{res.}}(L)/\mathrm{ad}L$ outer.

\begin{theorem}
Let $L$ is a restricted Lie superalgebra, then
$$\mathrm{Der}_{\mathrm{res.}}(L)/\mathrm{ad}L=H^{1}_{*}(L; L).$$
\end{theorem}
\begin{proof}

 (1) Let  $\varphi\in C_{*}^{1}(L; L)$. By the  equation (\ref{e2.5}) we know  that $\varphi$ is a $1$-dimensional restricted cocycle if and only if $d^{1}\varphi=0$  and $\mathrm{ind}^{1}(\varphi)=0.$
\begin{itemize}
  \item [(1.1)]  If $d^{1}\varphi=0$, then by the  equation (\ref{e2.1}) we have
    \begin{eqnarray*}
0=d^{1}\varphi(x, y)&=&-\varphi([x, y])+(-1)^{|x||\varphi|}[x, \varphi(y)]-(-1)^{|y|(|\varphi|+|x|)}[y, \varphi(x)]\\
&=&-\varphi([x, y])+(-1)^{|x||\varphi|}[x, \varphi(y)]+[\varphi(x), y]
  \end{eqnarray*}
  for all $x, y\in L.$ So $\varphi$ satisfies the  equation (\ref{e3.1}).
  \item [(1.2)] If $\mathrm{ind}^{1}(\varphi)=0$,  then by the  equation (\ref{e2.2}) we have
    \begin{eqnarray*}
    &&0=\mathrm{ind}^{1}(\varphi)(z)=\varphi(z^{[p]})-z^{p-1}\varphi(z)
  \end{eqnarray*}
 for all $z\in L_{\bar{0}}.$ So $\varphi$ satisfies the  equation (\ref{e3.2}).
\end{itemize}
In summary, $\varphi$ is a $1$-dimensional restricted  cocycle if and only if $\varphi\in \mathrm{Der}_{\mathrm{res.}}(L),$ that is
$$Z^{1}_{*}(L; L)=\mathrm{Der}_{\mathrm{res.}}(L).$$

(2) Let $\psi\in B^{1}_{*}(L; L)$, then there is $x\in C^{0}_{*}(L; L)\cong L$ such that $\psi=d^{0}_{*}(x).$ So by the  equations (\ref{e2.1}) and (\ref{e2.4}) we have
$$\psi(y)=d^{0}_{*}(x)(y)=(-1)^{|x||y|}[y, x]=-[x, y]$$
for all $y\in L.$ Therefore, $\psi=-\mathrm{ad}x.$ It follows that $B^{1}_{*}(L; L)=\mathrm{ad}L.$

The proof is complete.
\end{proof}

\begin{definition}Let $L$ be a restricted Lie superalgebra and superspaces   $M, N$ be two restricted $L$-modules. A restricted extension of $N$  by $M$ is an exact sequence
\begin{eqnarray}\label{e3.3}
 0\longrightarrow M\stackrel{\iota}{\longrightarrow} E\stackrel{\pi}{\longrightarrow} N \longrightarrow 0
\end{eqnarray}
of restricted $L$-modules and homomorphisms.
Two extensions of $N$ by $M$ are equivalent if there is an isomorphism of restricted $L$-modules $\sigma: E_{1}\longrightarrow E_{2}$ so that
\begin{eqnarray}\label{e3.7}
\begin{array}[c]{ccccccccc}
 0&\longrightarrow &M&\stackrel{\iota_{1}}{\longrightarrow}& E_{1}&\stackrel{\pi_{1}}{\longrightarrow}& N& \longrightarrow& 0\\
 &&\parallel&\circlearrowright&\downarrow\scriptstyle{\sigma}&\circlearrowright&\parallel&&\\
 0&\longrightarrow &M&\stackrel{\iota_{2}}{\longrightarrow}& E_{2}&\stackrel{\pi_{2}}{\longrightarrow}& N& \longrightarrow& 0
\end{array}
\end{eqnarray}
commutes, and we ask for a description of the set $\mathrm{Ext_{res.}}(N, M)$ of equivalence classes of extensions of $N$ by $M$.
\end{definition}

Let $L$ be a restricted Lie superalgebras and $M, N$ be two restricted $L$-modules. For $x\in L$, $\phi\in \mathrm{Hom}_{\mathbb{F}}(N, M)$ and $n\in N$ we define
\begin{eqnarray}\label{e3.5}
(x\phi)(n)=x\phi(n)-(-1)^{|\phi||x|}\phi(xn).
\end{eqnarray}
Then $\mathrm{Hom}_{\mathbb{F}}(N, M)$ is a restricted $L$-module and we have the following Theorem.
\begin{theorem}
Let $L$ be a restricted Lie superalgebras and $M, N$ be two restricted $L$-modules. Then $\mathrm{Ext_{res.}}(N, M)$ is in one to one correspondence $H_{*}^{1}(L; \mathrm{Hom}_{\mathbb{F}}(N, M))$. In particular,  $\mathrm{Ext_{res.}}(\mathbb{F}, M)$ is in one to one correspondence $H_{*}^{1}(L; M)$.
\end{theorem}
\begin{proof}
(1)Given an extension (\ref{e3.3}) of $N$ by $M$ and  choose a $\mathbb{F}$-linear map $\rho$ such that $\pi \rho=\mathrm{Id}_{N}.$ We define an element $\varphi_{\rho}\in C_{*}^{1}(L; \mathrm{Hom}_{\mathbb{F}}(N, M))$ by the formula
\begin{eqnarray}\label{e3.4}
 \varphi_{\rho}(x)(n)=(-1)^{|x||\rho|}x\rho(n)-\rho(xn) \;\mbox{for all}\;x \in L, n\in N,
\end{eqnarray}
that is $\varphi_{\rho}(x)=(-1)^{|x||\rho|}x\rho.$
 We claim that $d^{1}_{*}\varphi_{\rho}=(d^{1}\varphi_{\rho}, \mathrm{ind}^{1}(\varphi_{\rho}))=0$, that is $\varphi_{\rho}\in Z_{*}^{1}(L; \mathrm{Hom}_{\mathbb{F}}(N, M))$.
 \begin{itemize}
   \item [(1.1)] For all $x, y\in L$ and $n\in N$, by the  equations (\ref{e2.1}), (\ref{e3.5}) and (\ref{e3.4}) we have
   \begin{eqnarray*}
 d^{1}\varphi_{\rho}(x, y)(n)&=&-\varphi_{\rho}([x,y])(n)+(-1)^{|x||\varphi_{\rho}|}(x\varphi_{\rho}(y))(n)\\
 &&-(-1)^{|y|(|x|+|\varphi_{\rho}|)}(y\varphi_{\rho}(x))(n)\\
 &=&-(-1)^{(|x|+|y|)|\rho|}[x, y]\rho(n)+\rho([x,y]n)\\
 &&+(-1)^{|x||\varphi_{\rho}|}x\varphi_{\rho}(y)(n)-(-1)^{|y||x|}\varphi_{\rho}(y)(xn)\\
 &&-(-1)^{|y|(|x|+|\varphi_{\rho}|)}y\varphi_{\rho}(x)(n)+\varphi_{\rho}(x)(yn)\\
 &=&-(-1)^{(|x|+|y|)|\rho|}[x, y]\rho(n)+\rho([x,y]n)\\
 &&+(-1)^{|x||\varphi_{\rho}|}(-1)^{|y||\rho|}x y\rho(n)-(-1)^{|x||\varphi_{\rho}|}x\rho(yn)\\
 &&-(-1)^{|y||x|}(-1)^{|y||\rho|}y\rho(xn)+(-1)^{|y||x|}\rho(yxn)\\
 &&-(-1)^{|y|(|x|+|\varphi_{\rho}|)}(-1)^{|x||\rho|}y x\rho(n)+(-1)^{|y|(|x|+|\varphi_{\rho}|)}y\rho(xn)\\
 &&+(-1)^{|x||\rho|}x\rho(yn)-\rho(xyn)\\
  &=&[-(-1)^{(|x|+|y|)|\rho|}[x, y]\rho(n)+(-1)^{|x||\varphi_{\rho}|}(-1)^{|y||\rho|}x y\rho(n)\\
  &&-(-1)^{|y|(|x|+|\varphi_{\rho}|)}(-1)^{|x||\rho|}y x\rho(n)]\\
 &&+[\rho([x,y]n)+(-1)^{|y||x|}\rho(yxn)-\rho(xyn)]\\
 &&[-(-1)^{|x||\varphi_{\rho}|}x\rho(yn)+(-1)^{|x||\rho|}x\rho(yn)]\\
 &&[-(-1)^{|y||x|}(-1)^{|y||\rho|}y\rho(xn)+(-1)^{|y|(|x|+|\varphi_{\rho}|)}y\rho(xn)].
 \end{eqnarray*}
 Note that $|\varphi_{\rho}|=|\rho|$. Then  $d^{1}\varphi_{\rho}(x, y)(n)=0,$ that is $ d^{1}\varphi_{\rho}=0.$
   \item [(1.2)] For all $x\in L_{\bar{0}}$, by the  equations (\ref{e2.2}) and (\ref{e3.4}) we have $$\mathrm{ind}^{1}(\varphi_{\rho})(x)=\varphi_{\rho}(x^{[p]})-x^{p-1}\varphi_{\rho}(x)=x^{[p]}\rho-x^{p}\rho=0.$$
       Then $\mathrm{ind}^{1}(\varphi_{\rho})=0.$
 \end{itemize}

 (2)
 Conversely, if $\varphi\in Z_{*}^{1}(L; \mathrm{Hom}_{\mathbb{F}}(N, M))$, we construct an extension of $N$ by $M$ as follows. We set $E=N\oplus M$
 and
 \begin{eqnarray}\label{e3.6}
 x(n,m)=(xn, xm+(-1)^{|\varphi|(|x|+|n|)}\varphi(x)(n)),
 \end{eqnarray}
 where $x\in L$, $n\in N$ and $m\in M.$ We claim that $E$ is a restricted $L$-module and
 \begin{eqnarray*}
 0\longrightarrow M\hookrightarrow E\twoheadrightarrow N \longrightarrow 0
\end{eqnarray*}
is a   restricted extension of $N$  by $M$.
\begin{itemize}
  \item [(2.1)] Since  $\varphi\in Z_{*}^{1}(L; \mathrm{Hom}_{\mathbb{F}}(N, M))$, then by the  equation (\ref{e2.1}) we have
  \begin{eqnarray}\label{e3.7-1}
      \varphi([x,y])=(-1)^{|x||\varphi|}x\varphi(y)-(-1)^{|y|(|x|+|\varphi|)}y\varphi(x)
   \end{eqnarray}
  for all $x, y\in L$.       By  the  equations (\ref{e3.6}) and (\ref{e3.7-1})  we have
   \begin{eqnarray*}
 &&(xy-(-1)^{|x||y|}yx)(n,m)\\
 &=&x(yn, ym+(-1)^{|\varphi|(|y|+|n|)}\varphi(y)(n))\\
 &&-(-1)^{|x||y|}y(xn, xm+(-1)^{|\varphi|(|x|+|n|)}\varphi(x)(n))\\
 &=&(xyn, xym+(-1)^{|\varphi|(|y|+|n|)}x\varphi(y)(n)+(-1)^{|\varphi|(|x|+|y|+|n|)}\varphi(x)(yn))\\
 &&-(-1)^{|x||y|}(yxn, yxm+(-1)^{|\varphi|(|x|+|n|)}y\varphi(x)(n)+(-1)^{|\varphi|(|y|+|x|+|n|)}\varphi(y)(xn))\\
 &=&([x,y]n, [x,y]m+(-1)^{|\varphi|(|y|+|n|)}(x\varphi(y))(n)-(-1)^{|x||y|+|\varphi|(|x|+|n|)}(y\varphi(x))(n))\\
 &=&([x,y]n, [x,y]m+(-1)^{|\varphi|(|[x,y]|+|n|)}\varphi([x,y])(n))=[x,y](n,m)
\end{eqnarray*}
for all $n\in N$ and $m\in M.$
  \item [(2.2)]Since  $\varphi\in Z_{*}^{1}(L; \mathrm{Hom}_{\mathbb{F}}(N, M))$, then
\begin{eqnarray}\label{e3.7-2}
\varphi(x^{[p]})=x^{p-1}\varphi(x)
\end{eqnarray}
 for all $x\in L_{\bar{0}}.$        By the  equations (\ref{e3.6})  and (\ref{e3.7-2}) we have
 \begin{eqnarray*}
x^{[p]}(n,m)&=&(x^{[p]}n, x^{[p]}m+(-1)^{|\varphi||n|}\varphi(x^{[p]})(n))\\
 &=&(x^{p}n, x^{p}m+(-1)^{|\varphi||n|}x^{p-1}\varphi(x)(n))=x^{p}(n, m)
\end{eqnarray*}
for all $n\in N$ and $m\in M.$
\item [(2.3)]  Evidently
the canonical inclusion and projection are restricted Lie superalgebra homomorphisms with
this module structure so that we have an extension of $N$ by $M$.
\end{itemize}

(3)  We claim that $\mathrm{Ext_{res.}}(N, M)$ is in one to one correspondence $H_{*}^{1}(L; \mathrm{Hom}_{\mathbb{F}}(N, M))$.
\begin{itemize}
\item [(3.1)]   Suppose that $\varphi_{1}, \varphi_{2}\in Z^{1}_{*}(L; \mathrm{Hom}_{\mathbb{F}}(N, M))_{\theta}$, where $\theta\in \mathbb{Z}_{2},$ $E_{1}, E_{2}$  denote the corresponding extensions of $N$ by $M$. If  $f\in \mathrm{Hom}_{\mathbb{F}}(N, M)_{\theta}$ satisfies $d^{0}_{*}f=\varphi_{2}-\varphi_{1}$. We claim that $E_{1}$ and  $E_{2}$ are equivalent.  We define a map
   \begin{eqnarray}\label{e3.7-3}
   \sigma: E_{1}\longrightarrow E_{2}\;\;\; \sigma(n,m)=(n, m-(-1)^{\theta|n|}f(n))
   \end{eqnarray}
     for all $n\in N$ and $m\in M.$
     Clearly this map is an isomorphism of vector spaces making the diagram (\ref{e3.7}) commutes.
    \begin{itemize}
      \item [(3.1.1)] Let $x\in L$, $n\in N$ and $m\in M$. Then by the  equations (\ref{e2.4}), (\ref{e3.5}), (\ref{e3.6}) and (\ref{e3.7-3}) we have
       \begin{eqnarray*}
&&\sigma(x(n,m))\\
&=&\sigma(xn, xm+(-1)^{\theta(|x|+|n|)}\varphi_{1}(x)(n))\\
 &=&(xn, xm+(-1)^{\theta(|x|+|n|)}\varphi_{1}(x)(n)-(-1)^{\theta(|n|+|x|)}f(xn))\\
 &=&(xn, xm+(-1)^{\theta(|x|+|n|)}(\varphi_{2}(x)(n)-d^{0}_{*}f(x)(n))-(-1)^{\theta(|n|+|x|)}f(xn))\\
  &=&(xn, xm+(-1)^{\theta(|x|+|n|)}(\varphi_{2}(x)(n)-(-1)^{|x|\theta}(xf)(n))-(-1)^{\theta(|n|+|x|)}f(xn))\\
  &=&(xn, xm+(-1)^{\theta(|x|+|n|)}\varphi_{2}(x)(n)-(-1)^{\theta|n|}xf(n))\\
  &=&x(n, m-(-1)^{\theta|n|}f(n))\\
&=&x\sigma(n, m).
 \end{eqnarray*}
So  $\sigma$ is a $L$-module homomorphism.
      \item [(3.1.2)] Let $x\in L_{\bar{0}}$, $n\in N$ and $m\in M$. Then by the  equations (\ref{e2.4}), (\ref{e3.5}), (\ref{e3.6}) and (\ref{e3.7-3}) we have
       \begin{eqnarray*}
&&\sigma(x^{[p]}(n,m))\\
&=&\sigma(x^{[p]}n, x^{[p]}m+(-1)^{\theta|n|}\varphi_{1}(x^{[p]})(n))\\
&=&\sigma(x^{p}n, x^{p}m+(-1)^{\theta|n|}x^{p-1}\varphi_{1}(x)(n))\\
 &=&(x^{p}n, x^{p}m+(-1)^{\theta|n|}x^{p-1}\varphi_{1}(x)(n)-(-1)^{\theta|n|}f(x^{p}n))\\
 &=&(x^{p}n, x^{p}m+(-1)^{\theta|n|}x^{p-1}(\varphi_{2}(x)(n)-d^{0}_{*}f(x)(n))-(-1)^{\theta|n|}f(x^{p}n))\\
  &=&(x^{p}n, x^{p}m+(-1)^{\theta|n|}x^{p-1}\varphi_{2}(x)(n)-(-1)^{|n|\theta}(x^{p}f)(n))-(-1)^{\theta|n|}f(x^{p}n))\\
  &=&(x^{p}n, x^{p}m+(-1)^{\theta|n|}x^{p-1}\varphi_{2}(x)(n)-(-1)^{\theta|n|}x^{p}f(n))\\
  &=&x^{p}(n, m-(-1)^{\theta|n|}f(n))\\
&=&x^{p}\sigma(n, m).
 \end{eqnarray*}
So $\sigma$ is a homomorphism of restricted $L$-modules.
    \end{itemize}

In summary, $\sigma$ is an isomorphism of restricted $L$-modules. Therefore,  $E_{1}$ and  $E_{2}$ are equivalent.
 \item [(3.2)] Since $d^{0}_{*}=d^{0}$ and $C^{1}_{*}(L; \mathrm{Hom}_{\mathbb{F}}(N, M))=C^{1}(L; \mathrm{Hom}_{\mathbb{F}}(N, M)),$ it follows from the classical cohomology theory that   restricted cocycle whose cohomology class depends
only on the equivalence class of the extension.
 \end{itemize}
 The proof is complete.
\end{proof}
\begin{definition}
We   say that a
restricted Lie  superalgebra $K$ is strongly abelian if in addition to $[K, K]=0$, we also have  $K_{\bar{0}}^{[p]}=0.$
\end{definition}
\begin{definition}
Let $L$ be a restricted Lie superalgebras and $K$ be a strongly abelian restricted Lie superalgebras. A restricted extension of $L$  by $K$ is an exact sequence
\begin{eqnarray}\label{e3.8}
 0\longrightarrow K\stackrel{\iota}{\longrightarrow} E\stackrel{\pi}{\longrightarrow} L \longrightarrow 0
\end{eqnarray}
of restricted Lie superalgebras and homomorphisms.
Two extensions of $L$ by $K$ being  equivalent if there is an isomorphism of restricted Lie superalgebras $\sigma: E_{1}\longrightarrow E_{2}$ so that
\begin{eqnarray*}\label{e3.9}
\begin{array}[c]{ccccccccc}
 0&\longrightarrow &K&\stackrel{\iota_{1}}{\longrightarrow}& E_{1}&\stackrel{\pi_{1}}{\longrightarrow}& L& \longrightarrow& 0\\
 &&\parallel&\circlearrowright&\downarrow\scriptstyle{\sigma}&\circlearrowright&\parallel&&\\
 0&\longrightarrow &K&\stackrel{\iota_{2}}{\longrightarrow}& E_{2}&\stackrel{\pi_{2}}{\longrightarrow}& L& \longrightarrow& 0
\end{array}
\end{eqnarray*}
commutes.
\end{definition}
Since  $K$ is  a strongly abelian restricted Lie superalgebra and $\iota(K)\vartriangleleft E$, then  a
restricted extension of $L$ by $K$ gives  the structure of a $L$-module by the action $x\cdot k=[\tilde{x}, \iota(k)]$,
where  $x\in L,$ $k\in K$ and $\tilde{x}\in E$  is any element satisfying $\pi(\tilde{x})=x$.
Since $\pi(\tilde{x}^{[p]})=\pi(\tilde{x})^{[p]}=x^{[p]},$ so
\begin{eqnarray*}
x^{[p]}\cdot k=[\widetilde{x^{[p]}}, k]=[\tilde{x}^{[p]}, k]=(\mathrm{ad}\tilde{x})^{p}(k)=\underbrace{x\cdots x}_{p}\cdot k
\end{eqnarray*}
for all $x\in L_{\bar{0}}$ and $k\in K.$
Then $K$ is a restricted $L$-module.  We remark here that if $\iota(K)$ is contained in the center of $E$, then $K$ is a trivial
$L$-module. Such an extension is called central.
\begin{theorem}
Let $L$ be a restricted Lie superalgebras and $K$ be strongly abelian restricted Lie superalgebras. Then the set of equivalence classes of restricted center extensions of $L$ by $K$ is in one to one correspondence with $H^{2}_{*}(L; K)$.
In particular,     $H^{2}_{*}(L; \mathbb{F})$ is in one to one correspondence with equivalence classes of
$1$-dimensional restricted center extensions of $L$.
\end{theorem}
\begin{proof}
\begin{itemize}
  \item [(1)]
Given an extension (\ref{e3.8}) of $L$ by $K$ and  choose a $\mathbb{F}$-linear map $\rho$ such that $\pi \rho=\mathrm{Id}_{L}.$ Since $\pi$  is a homomorphism of  restricted Lie superalgebras, so
$$[\rho(x), \rho(y)]-\rho([x, y]), \rho(x)^{[p]}-\rho(x^{[p]})\in \mathrm{Ker}\pi=\mathrm{Im}\iota.$$
Then  we can define $\alpha_{\rho}: L\times L\longrightarrow K$ and $\beta_{\rho}: L_{\bar{0}}\longrightarrow K$ with the formula
\begin{eqnarray}\label{e3.8-1}
\alpha_{\rho}(x, y)=\iota^{-1}([\rho(x), \rho(y)]-\rho([x, y])),
\end{eqnarray}
\begin{eqnarray*}\label{e3.11}
\beta_{\rho}(z)=\iota^{-1}(\rho(z)^{[p]}-\rho(z^{[p]})),
\end{eqnarray*}
where $x, y\in L$ and $z\in L_{\bar{0}}.$
 Similar to \cite[Theorem 3.7]{e}   we  can prove that $\beta_{\rho}$ has  the $*$-property with
respect to $\alpha_{\rho}.$ Then $(\alpha_{\rho}, \beta_{\rho})\in C^{2}_{*}(L; K).$ We claim that $(\alpha_{\rho}, \beta_{\rho})\in Z^{2}_{*}(L; K)$.
  \begin{itemize}
    \item [(1.1)] Let $x_{1}, x_{2},x_{3}\in L.$ By  the equations (\ref{e2.1}) and (\ref{e3.8-1}) we have
     \begin{eqnarray*}
d^{2}\alpha_{\rho}(x_{1}, x_{2},x_{3})&=&-\alpha_{\rho}([x_{1},x_{2}], x_{3})+(-1)^{|x_{2}||x_{3}|}\alpha_{\rho}([x_{1},x_{3}], x_{2})\\
&&-(-1)^{|x_{1}|(|x_{2}|+|x_{3}|)}\alpha_{\rho}([x_{2},x_{3}], x_{1})\\
&=&\iota^{-1}(-[\rho([x_{1},x_{2}]), \rho(x_{3})]+(-1)^{|x_{2}||x_{3}|}[\rho([x_{1},x_{3}]), \rho(x_{2})]\\
&&-(-1)^{|x_{1}|(|x_{2}|+|x_{3}|)}[\rho([x_{2},x_{3}]), \rho(x_{1})]\\
&&+\rho([x_{1},x_{2}], x_{3})-(-1)^{|x_{2}||x_{3}|}\rho([x_{1},x_{3}], x_{2})\\
&&+(-1)^{|x_{1}|(|x_{2}|+|x_{3}|)}\rho([x_{2},x_{3}], x_{1})).
\end{eqnarray*}
Since $\iota^{-1}([\rho(x_{1}), \rho(x_{2})]-\rho([x_{1}, x_{2}]))\in K$ and $[\iota(K),E]=0,$ so
$$\iota^{-1}([[\rho(x_{1}), \rho(x_{2})], \rho(x_{3})])=\iota^{-1}([\rho([x_{1}, x_{2}]), \rho(x_{3})]).$$
Then
     \begin{eqnarray*}
&&d^{2}\alpha_{\rho}(x_{1}, x_{2},x_{3})\\
&=&\iota^{-1}(-[[\rho(x_{1}),\rho(x_{2})], \rho(x_{3})]+(-1)^{|x_{2}||x_{3}|}[[\rho(x_{1}),\rho(x_{3})], \rho(x_{2})]\\
&&-(-1)^{|x_{1}|(|x_{2}|+|x_{3}|)}[[\rho(x_{2}),\rho(x_{3})], \rho(x_{1})]\\
&&+\rho([[x_{1},x_{2}], x_{3}]-(-1)^{|x_{2}||x_{3}|}[[x_{1},x_{3}], x_{2}])+(-1)^{|x_{1}|(|x_{2}|+|x_{3}|)}[[x_{2},x_{3}], x_{1}]))\\
&=&0,
\end{eqnarray*}
that is $d^{2}\alpha_{\rho}=0.$
    \item [(1.2)] Since  $K$ is a trivial
$L$-module, then by equations (\ref{e2.3})  and (\ref{e3.8-1}) we have
    \begin{eqnarray*}
\mathrm{ind}^{2}(\alpha_{\rho},\beta_{\rho})(x, y)&=&\alpha_{\rho}(x, y^{[p]})-\sum_{i+j=p-1}(-1)^{i}y^{i}\alpha_{\rho} ([x,\underbrace{y,\ldots, y}_{j}],y)+x\beta_{\rho}(y)\\
    &=&\iota^{-1}([\rho(x), \rho(y^{[p]})]-\rho([x, y^{[p]}]))-\alpha_{\rho} ([x,\underbrace{y,\ldots, y}_{p-1}],y)\\
    &=&\iota^{-1}([\rho(x), \rho(y)^{[p]}]-\rho([x, y^{[p]}]))-\alpha_{\rho} ([x,\underbrace{y,\ldots, y}_{p-1}],y)\\
    &=&\iota^{-1}([\rho(x), \underbrace{\rho(y),\ldots,\rho(y)}_{p}]-\rho([x, \underbrace{y, \ldots, y}_{p}]))\\
    &&- \iota^{-1}([\rho([x,\underbrace{y,\ldots, y}_{p-1}]), \rho(y)]+\rho([x, \underbrace{y, \ldots, y}_{p}]))=0
    \end{eqnarray*}
  for all $x, y\in L_{\bar{0}}.$
  \end{itemize}
  By the equation (\ref{e2.6}) we have $d^{2}_{*}(\alpha_{\rho}, \beta_{\rho})=0.$
  \item [(2)] For each $ (\alpha, \beta)\in Z^{2}_{*}(L; K),$ one can  construct a restricted  central extension $L$ by $K$ as follows.
   We define $L_{\alpha, \beta}=K\oplus L$ as a vector space and we define the
Lie bracket and p-operator in $L_{\alpha, \beta}$ by the formula
 \begin{eqnarray}\label{e3.8-2}
 [(k_{1}, l_{1}), (k_{2}, l_{2})]=(\alpha(l_{1}, l_{2}), [l_{1}, l_{2}]),
     \end{eqnarray}
    \begin{eqnarray}\label{e3.8-3}
    (k, l)^{[p]}=(\beta(l), l^{[p]}),
        \end{eqnarray}
where  $k\in K_{\bar{0}}, k_{1}, k_{2}\in K$ and $l\in L_{\bar{0}}, l_{1}, l_{2}\in L.$
\begin{itemize}
  \item [(2.1)] The equation (\ref{e3.8-2})   is well known that the Jacobi identity    is equivalent to  $d^{2}\alpha=0.$
  \item [(2.2)] The equation (\ref{e3.8-3}) is a $p$-operator precisely because $\beta$ has the $*$-property with respect to $\alpha$ and $\mathrm{ind}^{2}(\alpha,\beta)=0.$
\end{itemize}
\item [(3)] Next we show that for $(\alpha_{1}, \beta_{1}), (\alpha_{2}, \beta_{2})\in Z^{2}_{*}(L; K),$  such that the central extensions $L_{\alpha_{1}, \beta_{1}}$ and $L_{\alpha_{2}, \beta_{2}}$ are equivalent, we have $(\alpha_{1}-\alpha_{2}, \beta_{1}-\beta_{2})\in B^{2}_{*}(L; K)$. Let  $\sigma$ be an  isomorphism  of restricted Lie superalgebras and $\rho_{1}, \rho_{2}$ be two $\mathbb{F}$-linear maps   such that
     \begin{eqnarray}\label{e3.12}
\begin{array}[c]{ccccccccc}
 0&\longrightarrow &K&\stackrel{\iota_{1}}{\longrightarrow}&L_{\alpha_{1}, \beta_{1}}&\stackrel{\pi_{1}}{\longrightarrow}& L& \longrightarrow& 0\\
 &&\parallel&\circlearrowright&\downarrow\scriptstyle{\sigma}&\circlearrowright&\parallel&&\\
 0&\longrightarrow &K&\stackrel{\iota_{2}}{\longrightarrow}& L_{\alpha_{2}, \beta_{2}}&\stackrel{\pi_{2}}{\longrightarrow}& L& \longrightarrow& 0
\end{array}
\end{eqnarray}
     commutes and  $\pi_{1} \rho_{1}=\pi_{2} \rho_{2}=\mathrm{Id}_{L}$. Since
    $$\pi_{2}(\sigma\rho_{1}-\rho_{2})=\pi_{2}\sigma\rho_{1}-\pi_{2}\rho_{2}=\pi_{1}\rho_{1}-\pi_{2}\rho_{2}=0,$$
    so $(\sigma\rho_{1}-\rho_{2})(L)\subseteq \mathrm{Ker}(\pi_{2})=\mathrm{Im}(\iota_{2}).$ Then  $\varphi=\iota_{2}^{-1}(\sigma \rho_{1}-\rho_{2})\in C^{1}_{*}(L; K).$  We claim that $d^{1}_{*}(\varphi)=(\alpha_{2}-\alpha_{1}, \beta_{2}-\beta_{1}).$
    \begin{itemize}
      \item [(3.1)] By proof of \cite[Lemma 4.1]{ik}, we have $d^{1}\varphi=\alpha_{2}-\alpha_{1}.$
      \item [(3.2)] For all $x\in L_{\bar{0}},$ the same
methods were given in the proof of \cite[Lemma 3.25]{e} show that
     \begin{eqnarray*}
     \mathrm{ind}^{1}(\varphi)(x)=(\beta_{2}-\beta_{1})(x).
     \end{eqnarray*}
    \end{itemize}

\item [(4)] Let $(\alpha_{1}, \beta_{1}), (\alpha_{2}, \beta_{2}),\in Z^{2}_{*}(L; K),$  such that $(\alpha_{2}, \beta_{2})-(\alpha_{1}, \beta_{1})\in B^{2}_{*}(L; K)$, that is there exists $\varphi\in C^{1}_{*}(L; K)$  satisfying $(\alpha_{2}, \beta_{2})-(\alpha_{1}, \beta_{1})=d^{1}_{*}\varphi.$   Now we prove that the extensions $L_{\alpha_{1}, \beta_{1}}$ and $L_{\alpha_{2}, \beta_{2}}$  are equivalent.  Let us define
    $\sigma: L_{\alpha_{1}, \beta_{1}}\longrightarrow L_{\alpha_{2}, \beta_{2}}$ by
\begin{eqnarray}\label{e3.9-1}
\sigma(k,l)=(k-\varphi(l), l)
\end{eqnarray}
    for all $k \in K$ and $l\in L.$
    It is clear that $\sigma$ is bijective and such that diagram (\ref{e3.12}) commutes. We claim $\sigma$ is a  homomorphism of restricted Lie superalgebras.
   \begin{itemize}
     \item [(4.1)] Let us check that $\sigma$ is a homomorphism of Lie superalgebras. By the equations (\ref{e3.8-2}) and (\ref{e3.9-1}) we have
     \begin{eqnarray*}
  &&[\sigma(k_{1},l_{1}),\sigma(k_{2},l_{2})]=[(k_{1}-\varphi(l_{1}), l_{1}), (k_{2}-\varphi(l_{2}), l_{2})]\\
  &=&(\alpha_{2}(l_{1}, l_{2}),[l_{1}, l_{2}])=(\alpha_{1}(l_{1}, l_{2})+d^{1}\varphi(l_{1}, l_{2}),[l_{1}, l_{2}])\\
  &=&(\alpha_{1}(l_{1}, l_{2})-\varphi[l_{1}, l_{2}],[l_{1}, l_{2}])=\sigma(\alpha_{1}(l_{1}, l_{2}), [l_{1}, l_{2}])\\
  &=&\sigma[(k_{1},l_{1}),(k_{2},l_{2})]
     \end{eqnarray*}
     for all   $ k_{1}, k_{2}\in K$ and $ l_{1}, l_{2}\in L.$
     \item [(4.2)] Next we show that  $\sigma$ is restricted. By the equations (\ref{e3.8-3}) and (\ref{e3.9-1}) we have
      \begin{eqnarray*}
 \sigma((k, l)^{[p]})&=&\sigma(\beta_{1}(l), l^{[p]})=(\beta_{1}(l)-\varphi(l^{[p]}), l^{[p]})\\
  &=&(\beta_{2}(l), l^{[p]})=(k-\varphi(l), l)^{[p]}=(\sigma(k, l))^{[p]}
     \end{eqnarray*}
      for all   $ k\in K_{\bar{0}}$ and $ l\in L_{\bar{0}}.$
\end{itemize}
   \end{itemize}
 The proof is complete.
\end{proof}

\section{Ordinary and restricted cohomology of the restricted model filiform Lie superalgebras}
In this Section, we determine both the $1, 2$-dimensional ordinary and restricted cohomology   of restricted model filiform  Lie superalgebra with coefficients   in the $1$-dimensional trivial module.
\subsection{Restricted model filiform Lie superalgebras}
Fix  a pair of positive integers $n, m$,  let $$L_{n,m}=\mathrm{span}_{\mathbb{F}}\{X_{1}, \ldots, X_{n}\mid Y_{1}, \ldots, Y_{m}\}$$ and Lie super-brackets are given by
     \begin{eqnarray*}
     [X_{1}, X_{i}]=X_{i+1},\;\;2\leq i\leq n-1, \;\; [X_{1}, Y_{j}]=Y_{j+1},\;\; 1\leq j\leq m-1
     \end{eqnarray*}
with all other brackets are zero. We call $L_{n,m}$ the model filiform Lie superalgebra and  $\{X_{1}, \ldots, X_{n}\mid Y_{1}, \ldots, Y_{m}\}$ the standard basis of $L_{n,m}.$  Next we show that $L_{p,p}$ admits the structure of a restricted Lie superalgebra.
If we let $\lambda=(\lambda_{1},\ldots, \lambda_{p})\in\mathbb{F}^{p}$, and set
$$(\sum_{k=1}^{p} a_{k}X_{k})^{[p]}=(\sum_{k=1}^{p} a_{k}^{p}\lambda_{k})X_{p},$$
where $a_{k}\in \mathbb{F},$
then by \cite{ef2}, we know that $(L_{p,p})_{\bar{0}}$ is a restricted Lie algebra.  Since
 \begin{eqnarray*}
[(\sum_{k=1}^{p} a_{k}X_{k})^{[p]}, Y]=[(\sum_{k=1}^{p} a_{k}^{p}\lambda_{k})X_{p}, Y]=0=\mathrm{ad}(\sum_{k=1}^{p} a_{k}X_{k})^{p} (Y),
     \end{eqnarray*}
where  $Y\in (L_{p,p})_{\bar{1}}$, then $L_{p,p}$ is a  restricted Lie superalgebra,  we denote by $L_{p,p}^{\lambda}.$

\subsection{Cochain Complexes with Trivial Coefficients}
For ordinary cohomology of  $L_{p,p}^{\lambda}$ with coefficients   in the $1$-dimensional trivial module, the relevant cochain superspaces (with dual bases) are:
  \begin{eqnarray*}
  C^{0}(L_{p,p}^{\lambda}; \mathbb{F})\;\;\;\; \Omega^{0}&=&\{1\},\\
C^{1}(L_{p,p}^{\lambda}; \mathbb{F})\;\;\; \Omega^{1}&=&\{X^{i}, Y^{j}\mid 1\leq i,j \leq p\},\\
C^{2}(L_{p,p}^{\lambda}; \mathbb{F})\;\;\; \Omega^{2} &=&\{X^{i,j}, X^{k}Y^{l}, Y^{s,t}\mid 1\leq i<j \leq p, 1\leq k,l\leq p, 1\leq s \leq t\leq p\},\\
C^{3}(L_{p,p}^{\lambda}; \mathbb{F})\;\;\; \Omega^{3}&=&\{X^{i,j,k}, X^{r,s}Y^{t}, X^{u}Y^{v,w}, Y^{x,y,z}\mid 1\leq i<j<k \leq p, 1\leq r< s\leq p, \\
  && 1\leq t, u\leq p, 1\leq v\leq w\leq p, 1\leq x \leq y \leq z\leq p\}.
  \end{eqnarray*}
By the equation (\ref{e2.1}) we have
\begin{eqnarray}\label{e4.12}
  d^{1}(X^{i})=-X^{1,i-1}, d^{1}(Y^{j})=-X^{1}Y^{j-1}, d^{1}(X^{1})=d^{1}(X^{2})=d^{1}(Y^{1})=0,  i\geq3, j\geq2,
  \end{eqnarray}
  \begin{eqnarray}\label{e4.13}
d^{2}(X^{1, i})&=&0,d^{2}(X^{2,3})=0, d^{2}(X^{2,j})=-X^{1,2,j-1}, i\geq 2, j\geq4,\\
d^{2}(X^{i,i+1})&=&-X^{1,i-1,i+1}, d^{2}(X^{j,k})=-X^{1,j-1,k}-X^{1, j,k-1}, i, j\geq3, k-1>j,\\\label{e4.14}
d^{2}(X^{1}Y^{i})&=&0,d^{2}(X^{2}Y^{1})=0, d^{2}(X^{2}Y^{j})=-X^{1,2}Y^{j-1},   i\geq 1, j\geq2,\\\label{e4.15}
d^{2}(X^{i}Y^{1})&=&-X^{1,i-1}Y^{1}, d^{2}(X^{i}Y^{j})=-X^{1,i-1}Y^{j}-X^{1, i}Y^{j-1},  i\geq 3, j\geq 2,\\\label{e4.16}
d^{2}(Y^{1,1})&=&0, d^{2}(Y^{1,i})=-X^{1}Y^{1,i-1}, d^{2}(Y^{i,i})=-2X^{1}Y^{i-1,i},  i \geq 2,\\\label{e4.17}
d^{2}(Y^{i,j})&=&-X^{1}Y^{i-1,j}-X^{1}Y^{i,j-1},  j> i\geq 2.
  \end{eqnarray}

For restricted cohomology  of   $L_{p,p}^{\lambda}$ with coefficients   in the $1$-dimensional trivial module, the relevant cochain superspaces (with   bases) are:
  \begin{eqnarray*}
  && C^{0}_{*}(L_{p,p}^{\lambda}; \mathbb{F})\;\;\;\; \Omega^{0},\\
  &&C^{1}_{*}(L_{p,p}^{\lambda}; \mathbb{F})\;\;\;\; \Omega^{1},\\
  &&C^{2}_{*}(L_{p,p}^{\lambda}; \mathbb{F})\;\;\; \;\{(0, X^{i}_{*}), (\varphi, \varphi_{*})\mid 1\leq i\leq p, \varphi\in \Omega^{2}\},
  \end{eqnarray*}
where $X^{k}_{*}(\sum_{i=1}^{p}a_{i}X_{i})=a_{k}^{p},$ $a_{k}\in \mathbb{F},$ $\varphi_{*}: (L_{p,p}^{\lambda})_{\bar{0}}\longrightarrow \mathbb{F}$ that vanishes on the basis and has the $*$-property w.r.t. $\varphi.$
\subsection{The cohomology $H^{1}(L_{p,p}^{\lambda}, \mathbb{F})$ and $H^{1}_{*}(L_{p,p}^{\lambda}, \mathbb{F})$}
In this subsection we study $1$-dimensional restricted cohomology of $L_{p,p}^{\lambda}$ with coefficients   in the $1$-dimensional trivial module.
\begin{theorem}
Let $\lambda\in \mathbb{F}^{p}.$ Then we have
$$H^{1}(L_{p,p}^{\lambda}, \mathbb{F})=H^{1}_{*}(L_{p,p}^{\lambda}, \mathbb{F})$$
and the classes of $\{X^{1}, X^{2}, Y^{1}\}$ form a basis.
\end{theorem}
\begin{proof}
It follows easily from (\ref{e4.12})  that $\mathrm{dim}(\mathrm{Ker}d^{1})=3$ and $\{X^{1}, X^{2}, Y^{1}\}$ is a basis for this kernel. So that
$$H^{1}(L_{p,p}^{\lambda}, \mathbb{F})=\mathrm{Ker}d^{1}=\mathrm{span}_{\mathbb{F}}\{X^{1}, X^{2}, Y^{1}\}.$$
Now, let $\varphi\in H^{1}(L_{p,p}^{\lambda}, \mathbb{F})$ and $\varphi=a_{1}X^{1}+a_{2}X^{2}+a_{3}Y^{1}$, where $a_{1}, a_{2}, a_{3}\in \mathbb{F}$. Then by the equation (\ref{e2.2}) we have
   \begin{eqnarray*}
  &&\mathrm{ind}^{1}(\varphi)(\sum_{i=1}^{p}b_{i}X_{i})=\varphi(\sum_{i=1}^{p}(b_{i}X_{i})^{[p]})=(a_{1}X^{1}+a_{2}X^{2}+a_{3}Y^{1})(\sum_{k=1}^{p} b_{k}^{p}\lambda_{k}X_{p})=0,
  \end{eqnarray*}
where  $b_{i}\in \mathbb{F}.$  So $H^{1}(L_{p,p}^{\lambda}, \mathbb{F})\subseteq H^{1}_{*}(L_{p,p}^{\lambda}, \mathbb{F}).$  The proof is complete.
\end{proof}

\subsection{The cohomology $H^{2}(L_{p,p}^{\lambda}, \mathbb{F})$ and $H^{2}_{*}(L_{p,p}^{\lambda}, \mathbb{F})$}
In this subsection we study $1$-dimensional restricted cohomology of $L_{p,p}^{\lambda}$ with coefficients   in the $1$-dimensional trivial module. For $i\in \{5, 7, \ldots, p+2\},$ $2\leq j\leq p$ and $k\in \{2, 4, \ldots, p+1\}$  we put
$$\varphi_{i}=\sum_{r=2}^{\lfloor \frac{i}{2}\rfloor}(-1)^{r}X^{r,i-r}, \psi_{j}=\sum_{s=2}^{j}(-1)^{s}X^{s}Y^{j-s+1}, \phi_{k}=\sum_{t=1}^{\frac{k}{2}-1}(-1)^{t}Y^{t,k-t}+(-1)^{\frac{k}{2}}\frac{1}{2}Y^{\frac{k}{2},\frac{k}{2}}.$$
Then we have the following Theorem.
\begin{theorem}\label{t4.2}
Let $\lambda\in \mathbb{F}^{p}.$ Then we have
$\mathrm{dim}H^{2}(L_{p,p}^{\lambda}, \mathbb{F})=2p+1$
and the cohomology classes of cocycles
$$\Omega^{4}=\{X^{1, p}, X^{1}Y^{p},  \varphi_{5}, \varphi_{7},\ldots, \varphi_{p+2}, \psi_{2}, \ldots, \psi_{p}, \phi_{2}, \phi_{4},\ldots, \phi_{p+1}\}$$
 form a basis.
\end{theorem}
\begin{proof}
It follows easily from the equation(\ref{e4.12})  that $\mathrm{dim}(\mathrm{Im}d^{1})=2p-3$ and
$$\Omega^{5}=\{X^{1,i}, X^{1}Y^{j}\mid 2\leq i \leq p-1, 1\leq j \leq p-1\}$$ is a basis for this image. By the equations (\ref{e4.13})--(\ref{e4.17}) we have
  \begin{eqnarray*}
  &&d^{2}\varphi_{i}=-X^{1,2,i-3}+\sum_{r=3}^{\lfloor \frac{i}{2}\rfloor-1}(-1)^{r}(-X^{1,r-1,i-r}-X^{1,r,i-r-1})-(-1)^{\lfloor \frac{i}{2}\rfloor}X^{1,\lfloor \frac{i}{2}\rfloor-1,i-\lfloor \frac{i}{2}\rfloor}=0,\\
  &&d^{2}\psi_{j}=-X^{1,2}Y^{j-2}+\sum_{s=3}^{j-1}(-1)^{s}(-X^{1,s-1}Y^{j-s+1}-X^{1,s}Y^{j-s})-(-1)^{j}X^{1,j-1}Y^{1}=0,\\
  &&d^{2}\phi_{k}=X^{1}Y^{1,k-2}+\sum_{t=2}^{\frac{k}{2}-1}(-1)^{t}(-X^{1}Y^{t-1,k-t}-X^{1}Y^{t,k-t-1})-(-1)^{\frac{k}{2}}X^{1}Y^{\frac{k}{2}-1,\frac{k}{2}}=0,
  \end{eqnarray*}
for $i\in \{5, 7, \ldots, p+2\},$ $2\leq j\leq p$ and $k\in \{2, 4, \ldots, p+1\}.$ Then
$$\Omega^{6}=\{X^{1,i}, X^{1}Y^{j}, \varphi_{5}, \varphi_{7},\ldots, \varphi_{p+2}, \psi_{i}, \phi_{2}, \phi_{4},\ldots, \phi_{p+1}\mid 2\leq i \leq p, 1\leq j \leq p\}\subseteq \mathrm{Ker}d^{2}.$$
Let $\varphi\in\mathrm{Ker}d^{2} $. By the equations  (\ref{e4.13})--(\ref{e4.17}),  we can suppose that
$$\varphi\in \mathrm{span}_{\mathbb{F}}\{X^{i,j}\mid 1\leq i<j \leq p\}\cup \mathrm{span}_{\mathbb{F}}\{X^{k}Y^{l}\mid 1\leq k, l\leq p\}\cup \mathrm{span}_{\mathbb{F}}\{Y^{s,t}\mid 1\leq s\leq t\leq p\}.$$
 \begin{itemize}
   \item [(1)] Suppose $\varphi\in \mathrm{span}_{\mathbb{F}}\{X^{i,j}\mid 1\leq i<j \leq p\}.$ Then it is clear that any cocycle element has to include either the basis element $X^{1,i}\;\;(2\leq i\leq p)$, and in this case this is a cocycle element, or it has to have one element of type $X^{2,j} \;(3\leq j \leq p)$ in the combination, and all those are combination of  $\varphi_{5}, \varphi_{7},\ldots, \varphi_{p+2}$.
   \item [(2)] Suppose $\varphi\in  \mathrm{span}_{\mathbb{F}}\{X^{k}Y^{l}\mid 1\leq k, l\leq p\}.$ Then it is clear that any cocycle element has to include either the basis element $X^{1}Y^{i} \;(1\leq i \leq p)$, and in this case this is a cocycle element, or it has to have one element of type $X^{2}Y^{j} \;(1\leq j \leq p)$ in the combination, and all those are combination of $\psi_{2}, \ldots, \psi_{p}$.
   \item [(3)] Suppose $\varphi\in  \mathrm{span}_{\mathbb{F}}\{Y^{s,t}\mid 1\leq s\leq t\leq p\}.$ Then it is clear that any cocycle element  has to have one  element of type $Y^{1,i} \;\;(1\leq i\leq p)$ in the combination, and all those are combination of $\phi_{2}, \phi_{4},\ldots, \phi_{p+1}$.
 \end{itemize}
The linear independence of the cocycle elements in $\Omega^{6}$ are clear. Then $\mathrm{Ker}d^{2}=\mathrm{span}_{\mathbb{F}}\Omega^{6}$ and $\mathrm{dim}\mathrm{Ker}d^{2}=4p-2.$ The proof is complete.
\end{proof}
\begin{theorem}
Let $\lambda=0.$ Then we have
$\mathrm{dim}H_{*}^{2}(L_{p,p}^{\lambda}, \mathbb{F})=3p+1$
and the cohomology classes of cocycles
$$\{(\varphi, \varphi_{*}), (0, X^{i}_{*})\mid \varphi\in \Omega^{4},  1\leq i \leq p\}$$
 form a basis.
\end{theorem}
\begin{proof}
\begin{itemize}
  \item [(1)] If $\lambda=0,$ then $\mathrm{ind}^{2}=0.$ So for  that every $2$-cocycle $\varphi\in C^{2}(L_{p,p}^{\lambda}, \mathbb{F})$, we have $d^{2}_{*}(\varphi, \varphi_{*})=(d^{2}\varphi, \mathrm{ind}^{2}\varphi_{*})=0,$ that is $(\varphi, \varphi_{*})\in \mathrm{Ker}d^{2}_{*},$ where $\varphi_{*}: (L_{p,p}^{\lambda})_{\bar{0}}\longrightarrow \mathbb{F}$ that vanishes on the basis and has the $*$-property with respect to $\varphi.$ The proof of Theorem \ref{t4.2} shows that linearly independent subset
$$\{(\varphi, \varphi_{*}), (0, X^{i}_{*})\mid \varphi\in \Omega^{6}, \mid 1\leq i \leq p\}$$
is a basis  of $\mathrm{Ker}d^{2}_{*}.$
  \item [(2)] Let $\psi\in C^{1}_{*}(L_{p,p}^{\lambda}, \mathbb{F}).$ Then $d^{1}_{*}(\psi)=(d^{1}\psi, \mathrm{ind}^{1}(\psi)),$ where $\mathrm{ind}^{1}(\psi)(x)=\psi(x^{[p]})$ for all $x\in (L_{p,p}^{\lambda})_{\bar{0}}.$ Since $\lambda=0,$ so $\mathrm{ind}^{1}(\psi)=\psi_{*}.$ It follows easily from the proof of Theorem \ref{t4.2}
      $$\mathrm{Im}d^{1}_{*}=\mathrm{span}_{\mathbb{F}}\{(\varphi, \varphi_{*})\mid \varphi\in \Omega^{5}\}.$$
\end{itemize}
The proof is complete.
\end{proof}

\begin{theorem}
Let $\lambda\neq0.$ Then we have
$\mathrm{dim}H_{*}^{2}(L_{p,p}^{\lambda}, \mathbb{F})=3p-1$
and the cohomology classes of cocycles
$$\{(\varphi, \varphi_{*}), (0, X^{i}_{*})\mid \varphi\in \Omega^{4}-\{X^{1, p},  \varphi_{p+2}\}, 1\leq i \leq p\}$$
 form a basis.
\end{theorem}
\begin{proof}
\begin{itemize}
  \item [(1)] Let $(\varphi, \omega)\in C^{2}_{*}(L_{p,p}^{\lambda}, \mathbb{F})$ and
  $$\varphi=\sum_{1\leq i<j \leq p}a_{i,j}X^{i,j}+\sum_{1\leq k, l\leq p}b_{k,l} X^{k}Y^{l}+\sum_{1\leq s\leq t\leq p} c_{s,t}Y^{s,t}\in C^{2}(L_{p,p}^{\lambda}, \mathbb{F}),$$
where $a_{ij}, b_{k,l}, c_{s,t}\in \mathbb{F}.$ By the equation (\ref{e2.3}) we have
  $$\mathrm{ind}^{2}(\varphi,\omega)(X_{i}, X_{j})=\varphi(X_{i}, X_{j}^{[p]})=\lambda_{j}a_{i,p}$$
  for all $1\leq i, j \leq p.$
  Since $\lambda\neq 0,$ so $d^{2}_{*}(\varphi, \omega)=(d^{2}\varphi, \mathrm{ind}^{2}(\varphi,\omega))=0$ if and only if $\varphi\in \mathrm{Ker}d^{2}$ and $a_{1,p}=\cdots= a_{p-1,p}=0$. This observation, together with the proof of Theorem \ref{t4.2}, proves the following
  $$\{(\varphi, \varphi_{*}), (0, X^{i}_{*})\mid \varphi\in \Omega^{6}-\{X^{1, p}, \varphi_{p+2}\} \mid 1\leq i \leq p\}$$
  is a basis  of $\mathrm{Ker}d^{2}_{*}.$
  \item [(2)] By the equation (\ref{e2.5}) we have
  \begin{eqnarray*}
  &&d^{1}_{*}(X^{1})=d^{1}_{*}(X^{2})=(0, 0), d^{1}_{*}(X^{i})=(-X^{1,i-1},(-X^{1,i-1})_{*}), 3\leq i\leq p-1,\\
  && d^{1}_{*}(X^{p})=(-X^{1,p-1},(-X^{1,p-1})_{*}+\sum_{k=1}^{p}\lambda_{k}X^{k}_{*}),\\
   &&d^{1}_{*}(Y^{1})=(0, 0), d^{1}_{*}(Y^{j})=(-X^{1}Y^{j-1},(-X^{1}Y^{j-1})_{*}), 2\leq j\leq p.
  \end{eqnarray*}
  Then  $$\mathrm{Im}d^{1}_{*}=\mathrm{span}_{\mathbb{F}}\{(\varphi, \varphi_{*}), (X^{1,p-1}, (X^{1,p-1})_{*}-\sum_{k=1}^{p}\lambda_{k}X^{k}_{*})\mid \varphi\in \Omega^{5}-\{X^{1, p-1}\}\}.$$
\end{itemize}
The proof is complete.
\end{proof}

\end{document}